\documentclass[reqno]{amsart}

\usepackage{amsfonts,amsthm,amsmath,amssymb,amscd,mathrsfs,color}
\usepackage{graphicx}
\usepackage{url}
\usepackage{dsfont}
\usepackage{MnSymbol}
\newtheorem{theorem}{Theorem}
\newtheorem{proposition}[theorem]{Proposition}
\newtheorem{lemma}[theorem]{Lemma}
\newtheorem{definition}{Definition}
\newtheorem{corollary}{Corollary}
\newtheorem{example}{Example}
\newtheorem{remark}{Remark}
\newtheorem{notation}{Notation}
\newtheorem{cor}{Corollary}

\newcommand{\bR}{\mathbb{R}}
\newcommand{\bC}{\mathbb{C}}
\newcommand{\ee}{\end{equation}}

\newcommand {\ga}{\gamma}
\newcommand {\Ga}{\Gamma}

\newcommand {\bCP} {\mathbb {CP}}

\newtheorem{thm}{Theorem}[section]

\theoremstyle{definition}

\theoremstyle{remark}
\newtheorem{rem}[thm]{Remark}
\theoremstyle{definition}

\theoremstyle{definition}

\theoremstyle{definition}

\begin{document}
          \numberwithin{equation}{section}

          \title[Plane rectifiable curves: old and new] 
          {Plane rectifiable curves: old and new} 

          \author[B.~Shapiro]{Boris Shapiro}
\address{Department of Mathematics, Stockholm University, SE-106 91, Stockholm,
            Sweden}
\email{shapiro@math.su.se}

\author[G.~Tahar]{Guillaume Tahar}
\address{Beijing Institute of Mathematical Sciences and Applications, Huairou
District, Beijing, China}

\email{guillaume.tahar@bimsa.cn}

\begin{abstract}
In this note we recall the classical notion of an algebraically rectifiable plane curve going back to J.~A.~Serret, E.~Laguerre and G.~Humbert. We provide new criteria of algebraic rectifiability, relate this notion to quadratic differentials, and generalize it to differentials of higher order.
\end{abstract}

\maketitle

\section{Introduction}
\subsection{Historical background}  
One of the fundamental problems in mathematics considered since the time of Archimedes and Apollonius is calculation of the length of an arc of a curve in two and three dimensions. In the 17-th century a  great progress in this area has been achieved  by such celebrities as E.~Torricelli, J.~Wallis, Ch.~Wren, G.~Leibniz, H.~van Heuraut, P. de Fermat, see e.g. \cite {Ka}.

 However there are not that many cases in which  the arc length can be found explicitly and typically, for a (real) algebraic curve, its arc length does not  depend algebraically on the endpoints of the arc. The fundamental example illustrating this phenomenon is  the  length of an arc of an ellipse whose study  led to the development of the theory of elliptic functions by C.~G.~Jacobi. 
  Apparently, historically first example of an algebraic curve whose arc length is an algebraic function of the endpoints is the semicubic parabola $y^2=k^2x^3$ where $k$ is a fixed real number.  It was discovered by W.~Niele in mid 17-th century and provides one of the few explicit examples of calculation of the length of a curve in the undergraduate calculus course.

 In the present paper we will mainly be interested in  the notion of an \emph{algebraically rectifiable curve} which is a (real) algebraic curve whose arc length  is an algebraic function of the endpoints. It is not quite clear who actually introduced it, but G.~Humbert \cite{Hu} quotes  J.~A.~Serret whose textbook "Cours d'alg\'ebra sup\'eriore" was very popular in the 1860's and later.
E.~Laguerre was definitely familiar with this notion and, in our opinion, it might as well be that I.~Newton, G.~Leibniz and L.~Euler thought about this problem.  In any case, Newton considered rather similar  algebraically integrable curves, comp. \cite{AV}.

The natural connection between the algebraic rectifiability and  evolutes was discovered in \cite {Hu}. (Originally, evolutes and involutes were in details studied by Ch.~Huygens, see \cite {Huy}. For the basic notions and results on evolutes see e.g. \cite {Hi}.) One of the fundamental results of \cite {Hu}, see p. 136, claims the following.

\begin{proposition}\label{prop:Hu} A plane algebraic curve is algebraically rectifiable if and only if it is the evolute of an algebraic curve.
\end{proposition}

Description of special classes of algebraically rectifiable curves can be found in  \cite{SaFa}. 

\medskip
\noindent

\subsection{Our results}
The aim of this paper is to recast the classical problem of algebraic
rectifiability in the language of meromorphic differentials on the
normalization of an algebraic curve.  To an algebraic plane curve
$\Gamma$ we attach the quadratic differential
\[
        ds^2=dx^2+dy^2
\]
on its normalization.  This elementary object records the algebraic
behaviour of the Euclidean length element, and its divisor is controlled by
isotropic tangencies and by the branches at infinity.  In particular, for a
generic rational curve of degree $d$ the zero divisor has degree $4d-4$,
while the poles occur at the points at infinity, each with order four.
This gives a useful bridge between classical projective geometry and the
strata of meromorphic quadratic differentials.

A first contribution is an intrinsic realization criterion.  Given a compact
real Riemann surface $(X,\rho)$ and a $\rho$-real quadratic differential
$q$, we show that $q$ is of the form $dx^2+dy^2$ for a real meromorphic
plane map if and only if it admits an exact admissible splitting
\[
        q=\omega\,\rho^*\overline{\omega},
        \qquad  \omega=dz .
\]
Thus the realization problem separates into a divisor-theoretic condition
and a period condition.  We also prove the corresponding uniqueness
statement: once the splitting, equivalently the divisor of $dz$, is fixed,
the realization is unique up to a direct Euclidean motion, and up to a
Euclidean similarity if the differential is multiplied by a positive
constant.

A second main point is the exactness criterion for higher order
differentials.  For a meromorphic $k$-differential $\phi$ on a compact
Riemann surface, we prove that exactness of the canonical $k$-cover is
equivalent to algebraicity of the Abelian integral $\int\phi^{1/k}$, and
also to the representation
\[
        \phi=f\left(\frac{df}{f}\right)^k
\]
for a meromorphic function $f$.  This gives a convenient uniform language
for rectifiability and for its affine analogue.  In the quadratic case it
recovers the algebraic arc-length condition.

We then relate this differential viewpoint to Humbert's classical theorem.
We prove that for a non-linear plane algebraic curve the following
conditions are equivalent: algebraic rectifiability, exactness of the
canonical double cover of $ds^2$, algebraicity of the involutes, and being
the evolute of an algebraic curve.  Thus Humbert's evolute criterion is
precisely the geometric form of the exactness criterion for the arc-length
quadratic differential.

The paper also contains several consequences and examples.  We show that a
non-linear rational real plane curve whose arc-length quadratic differential
has only two singularities must be a circle, up to Euclidean similarity and
reparametrization.  We compute the corresponding differentials for conics,
the Bernoulli lemniscate, the semicubic parabola and several rational
examples, and locate them in the appropriate strata.  Finally, we discuss
the analogous construction in special affine geometry, where the Euclidean
quadratic differential is replaced by the affine cubic differential
$(\gamma'\wedge\gamma'')dt^3$, and obtain the parallel exactness criterion
for algebraic affine involutes.  Conceptually, the paper studies when the Euclidean metric restricted to an algebraic curve becomes algebraically integrable.

\medskip
Let us also mention several closely related modern developments. Farouki and Sakkalis proved that, except for straight lines, real rational curves cannot be parametrized by rational functions of arc length, and subsequently studied algebraically rectifiable parametric curves, including low-degree classifications and links with Pythagorean-hodograph curves \cite{FaSa91,SaFa}.  The latter theory has become a substantial subject in geometric design; see, for example, the survey \cite{KoLa14} and the monograph \cite{Fa08}.  More recently, Solynin and Solynin revisited the interpretation of real algebraic curves via meromorphic quadratic differentials, following the circle of ideas of Langer and Singer \cite{LaSi1,SoSo22}.

\medskip 
The structure of this note is as follows. In Section~\ref{sec:arc}, we introduce the arc length form and the associated quadratic differential and record their basic divisor properties. In Section~\ref{sec:which} we study which quadratic differentials on a given Riemann surface can be realized as arc length differentials and how unique such a realization is. Section~\ref{sec:exact} discusses exact quadratic and higher order differentials, which provide a natural language for algebraic rectifiability. Section~\ref{sec:main} provides characterizations of plane rectifiable algebraic curves. Section~\ref{sec:examp} gives examples, including conics, the Bernoulli lemniscate and semicubic parabolas. Section~\ref{sec:aff} sketches the analogous questions in affine differential geometry, and the appendix recalls the Pl\"ucker theorems used for generic curves and their evolutes.


\section{Arc length $1$-form and quadratic differential} \label{sec:arc}

The following notions appeared in different contexts and should be better known.  We  rediscovered them as well before finding   relevant discussions in \cite{LaSi1,LaSi2,LaSi3, La}. The latter texts contain connections of these objects  with foci, the Schwarz derivative, the matrix range, etc.

\begin{definition}  For a plane algebraic curve $\Ga \subset \bC^2$ with the affine coordinates $(x,y)$ given by the equation $P(x,y)=0$, we define its {\bf arc length $1$-form}  as:  $$ds=\frac{\sqrt{(P_x^{\prime})^2+(P_y^\prime)^2}}{P^\prime_y}dx$$
 and we define its  {\bf arc length quadratic differential} as: $$ds^2:=(ds)^{\otimes 2}=\frac{(P_x^{\prime})^2+(P_y^\prime)^2}{(P^\prime_y)^2}dx^2.$$
  \end{definition}

  \begin{notation}
 Given a plane algebraic curve $\Ga \subset \bC^2$, denote by $\overline \Ga \subset \bC P^2$ its projective closure and by $\widetilde \Ga$ the normalization of $\overline \Ga$. Denote by $\Pi: \widetilde \Ga \dashedrightarrow \overline\Ga\subset \bC P^2$ the map sending the normalization to the projective closure.
  \end{notation}

\begin{lemma}\label{lm:basic}  In the above notation, the differential $ds^2$ is a well-defined meromorphic quadratic differential on the Riemann surface $\widetilde \Ga$.   The differential  $ds^2$  is the pullback of the standard meromorphic quadratic differential $dS^2:=dx^2+dy^2$  defined on $\bC P^2$ under the birational map $\Pi$.
\end{lemma}
\begin{proof}
The assertion follows by computing the pullback of the Euclidean quadratic form along a local parametrization of the normalization. On the smooth affine locus, implicit differentiation gives $dy=-(P_x/P_y)dx$, hence
\[
 dx^2+dy^2=\frac{P_x^2+P_y^2}{P_y^2}dx^2.
\]
This expression is independent of the local affine chart and therefore extends meromorphically to the normalization.
\end{proof}

\begin{remark} The number of  poles of $ds^2$ minus the number of its zeros equals $2 \chi(\widetilde\Ga)$. (We count both poles and zeros with their multiplicities).
\end{remark}

Let us now describe the singular points of  $ds^2$.
Denote by $D_0$ (resp. $D^\infty$)  the divisor of its zeros (resp. poles) on the normalization $\widetilde \Ga$.

\medskip
We need the following classical definition (see short introduction in \cite{PiSh}, \cite{Ri} and further references therein). One can easily check  that any circle in $\bR^2$ (or $\bC^2$) with coordinates $(x,y)$ intersects the  line at infinity at  two points $I$ and $J$ given by $(1,i,0)$ and $(1,-i,0)$ respectively in the homogeneous coordinates $(X:Y:Z)$ such that $x=X/Z$ and $y=Y/Z$.  (Here $i$ is the imaginary unity).

\medskip 
The points $I$ and $J$ are called  the \emph{circular points at infinity}. Lines in $\bC P^2\supset \bC^2$ passing though  either $I$ or $J$, and different from the line at infinity, are called \emph{circular} or \emph{isotropic lines}.  In the 1830s generalizing the properties of foci of a quadric,  J.~Pl\"ucker introduced the notion of a \emph{focus} of an arbitrary algebraic curve $\ga \subset \bC P^2$ as the intersection point of two tangents to $\ga$ one of which passes through $I$ and another through $J$. If neither $I$ nor $J$ lie on $\ga$, then its foci are called \emph{ordinary} and their number (counting multiplicity)  equals to the square of the class of $\ga$. For a real curve $\ga$,  only class-many of its foci can be real, the remaining are complex. (See more details in \cite{Ri}.) A tangent line to $\ga$ passing through either $I$ or $J$ will be called \emph{circular}. Denote by $\mathfrak D_\ga\subset \ga$ the divisor of all points of contact with $\ga$ of the circular tangent lines to $\ga$. Typically, $\mathfrak D$ consists of twice the class of $\ga$ smooth points of $\ga$.

Notice that the tangent line is well-defined at each smooth point of $\ga$. Lines passing through singular points of $\ga$ are called \emph{tangent} if they represent points of the dual curve $\ga^\ast$.

\begin{proposition}\label{prop:divisor} For a  generic plane algebraic curve $\Ga\subset \bC^2$, the divisor  $D_0$ of zeros of $ds^2$ coincides with the pullback  $\Pi^{-1}(\mathfrak D_\Ga)$  and the divisor $D^\infty$ of poles of $ds^2$ coincides with $\Pi^{-1}(\Ga^\infty)$ where $\Ga^\infty$ is the divisor of $\overline \Ga$  at infinity. Thus, if $\Ga$ is a generic curve of degree $d$, then $|D_0|=2d(d-1)$ and $D^\infty$ consists of $d$ distinct points with certain multiplicity, see below. For a generic rational curve $\Ga$ of degree $d$, $|D_0|=4d-4$ and $D^\infty$ consists of $d$ distinct points where each pole has  multiplicity $4$.
\end{proposition}

\begin{proof}
For a local parametrization $t\mapsto (x(t),y(t))$ of the normalization one has
\[
 ds^2=\bigl(x'(t)^2+y'(t)^2\bigr)dt^2.
\]
Thus zeros occur exactly when the tangent direction is isotropic, i.e. when the tangent line passes through one of the circular points at infinity. For a generic curve all such contacts are simple, and their number equals twice the class of the curve. Poles occur at the points lying over the line at infinity. In the rational generic case the class is $2d-2$, giving $4d-4$ zeros, and a local parameter at each of the $d$ branches at infinity gives a pole of order $4$.
\end{proof}

\begin{corollary} \rm{(i)} For a generic $\Ga$ of degree $d$, the number of poles of $ds^2$ (with multiplicities) minus the number of zeros equals $4d - 2d(d-1)=2 \chi (\Ga)=2(2-2g(\Ga))=4 - 2(d-1)(d-2)$. 

\medskip
\noindent
\rm{(ii)}  For a generic rational $\Ga$ of degree $d$, the number of poles of $ds^2$ (with multiplicities) minus the number of zeros equals $4d - (4d-4)=2 \chi (\Ga)=4$.
\end{corollary}

\begin{proof} Straightforward calculation. \end{proof}

\section{Which quadratic differentials can be induced as arc length?}\label{sec:which}

Given a compact Riemann surface $X$, we now ask which quadratic differentials
on $X$ can be induced as arc length differentials from
$dS^2=dx^2+dy^2$ under birational realizations
\[
   \Pi:X\dashrightarrow \bC P^2
\]
of $X$ as a plane algebraic curve.  In this form the question has two parts:
existence and uniqueness.  The point of this subsection is that the correct
object is not only the quadratic differential $q$, but also a splitting
\[
   q=\omega\,\eta
\]
into two conjugate meromorphic $1$-forms.  In the plane one has
\[
   z=x+iy,\qquad w=x-iy,\qquad q=dx^2+dy^2=dz\,dw,
\]
so the two factors are the pullbacks of the two isotropic differentials
$dz$ and $dw$.

We shall use the following terminology.  Let $X$ be equipped with an
anti-holomorphic involution $\rho$, and let $q$ be a non-zero meromorphic
quadratic differential satisfying
\[
   \rho^*\overline q=q .
\]
A divisor $D$ on $X$ will be called an admissible half-divisor for $q$ if
\[
   D+\rho(D)=\operatorname{div}(q).
\]
An admissible half-divisor is called exact if there exists a meromorphic
$1$-form $\omega$ with
\[
   \operatorname{div}(\omega)=D,\qquad
   q=\omega\,\rho^*\overline{\omega},
\]
and all periods of $\omega$ vanish.  Equivalently, $\omega=dz$ for a
meromorphic function $z$ on $X$.

\begin{theorem}[Existence criterion]\label{thm:existence-criterion}
Let $(X,\rho)$ be a compact real Riemann surface and let $q$ be a non-zero
$\rho$-real meromorphic quadratic differential on $X$. Then $q$ is induced by
a real meromorphic plane map
\[
   X\dashrightarrow \bC^2,\qquad p\mapsto (x(p),y(p)),
\]
in the sense that
\[
   q=dx^2+dy^2,
\]
if and only if $q$ admits an exact admissible half-divisor.  More explicitly,
this means that there exists a meromorphic $1$-form $\omega$ such that
\[
   q=\omega\,\rho^*\overline{\omega}
\]
and $\omega$ has zero periods.

If $\omega=dz$, the corresponding real plane map is
\[
   x={1\over 2}\bigl(z+\overline{z\circ\rho}\bigr),\qquad
   y={1\over 2i}\bigl(z-\overline{z\circ\rho}\bigr).
\]
It is birational onto its image precisely when the two meromorphic functions
$x$ and $y$ generate the full function field $\bC(X)$.
\end{theorem}

\begin{proof}
Assume first that $q$ comes from a real plane map.  Put
$z=x+iy$ and $w=x-iy$.  Since the map is real with respect to $\rho$, one has
$w=\overline{z\circ\rho}$.  Therefore
\[
   q=dx^2+dy^2=dz\,dw=dz\,\rho^*\overline{dz}.
\]
Thus $D=\operatorname{div}(dz)$ is an admissible half-divisor.  Moreover
$dz$ is exact, hence has zero periods.

Conversely, suppose that
\[
   q=\omega\,\rho^*\overline{\omega}
\]
with $\omega$ exact.  Choose a meromorphic primitive $z$ of $\omega$, so that
$\omega=dz$, and set
\[
   w=\overline{z\circ\rho},\qquad
   x={z+w\over 2},\qquad y={z-w\over 2i}.
\]
Then $x$ and $y$ are real meromorphic functions on $X$, and
\[
   dx^2+dy^2=dz\,dw
      =\omega\,\rho^*\overline{\omega}=q.
\]
This proves existence of a real meromorphic plane map.  Finally, a meromorphic
map to the plane is birational onto its image exactly when its coordinate
functions generate the function field of $X$; otherwise the map factors
through the intermediate curve with function field $\bC(x,y)$.
\end{proof}

\begin{remark}\label{rem:existence-obstructions}
The theorem separates the two obstructions to existence.  The first is a
divisor-theoretic obstruction: the divisor of $q$ has to split as
$D+\rho(D)$.  In particular, at a fixed point of $\rho$ all orders in
$\operatorname{div}(q)$ must be even.  The second is a period obstruction: the
corresponding meromorphic $1$-form $\omega$ must be exact.  The latter
condition includes, in particular, vanishing of all residues at the poles of
$\omega$.
\end{remark}

\begin{proposition}[Uniqueness for a fixed splitting]\label{prop:unique-fixed-splitting}
Let $q$ be as in Theorem~\ref{thm:existence-criterion}. Suppose that two real
meromorphic maps
\[
   z_j=x_j+i y_j,\qquad w_j=x_j-i y_j,\qquad j=1,2,
\]
induce the same quadratic differential
\[
   q=dz_1\,dw_1=dz_2\,dw_2,
\]
where $w_j=\overline{z_j\circ\rho}$.  If
\[
   \operatorname{div}(dz_1)=\operatorname{div}(dz_2),
\]
then the two realizations differ by a direct Euclidean motion.  More precisely,
there exist $\theta\in\mathbb R$ and $c\in\mathbb C$ such that
\[
   z_2=e^{i\theta}z_1+c.
\]
If the induced quadratic differentials are proportional by a positive
constant, then the conclusion is the same with a Euclidean similarity in place
of a Euclidean motion.
\end{proposition}

\begin{proof}
The quotient $dz_2/dz_1$ is a nowhere vanishing meromorphic function on the
compact Riemann surface $X$, hence is constant.  Thus
\[
   dz_2=\lambda dz_1
\]
for some $\lambda\in\mathbb C^*$.  The reality condition gives
$dw_2=\overline\lambda dw_1$.  Since the quadratic differentials are equal,
\[
   dz_2\,dw_2=|\lambda|^2 dz_1\,dw_1=dz_1\,dw_1,
\]
so $|\lambda|=1$.  Hence $\lambda=e^{i\theta}$, and integration gives
$z_2=e^{i\theta}z_1+c$.  The proportional case is identical and gives
$|\lambda|$ equal to the similarity ratio.
\end{proof}

\begin{corollary}\label{cor:finite-realizations-fixed-q}
For a fixed pair $(X,q)$ satisfying the hypotheses above, the realizations of
$q$ are classified, up to direct Euclidean motions, by exact admissible
half-divisors $D$ such that the associated primitive $z$ gives
$\bC(x,y)=\bC(X)$.  In particular, for fixed $q$ there are only finitely many
possible divisor splittings, and for each such splitting there is at most one
realization modulo direct Euclidean motions.
\end{corollary}

\begin{proof}
Every realization gives the divisor $D=\operatorname{div}(dz)$, and this
$D$ is an exact admissible half-divisor by Theorem~\ref{thm:existence-criterion}.
Conversely, every exact admissible half-divisor gives a real meromorphic plane
map.  Proposition~\ref{prop:unique-fixed-splitting} shows uniqueness for a
fixed $D$.  Since the support and multiplicities of $\operatorname{div}(q)$ are
finite, there are only finitely many decompositions
$D+\rho(D)=\operatorname{div}(q)$.
\end{proof}

\begin{lemma}\label{lm:res}
Let $p$ be a smooth point of $\Ga$, and assume that a line through one of the circular points at infinity has ordinary tangency with $\Ga$ at $p$. Then the corresponding zero of the arc length quadratic differential $ds^2$ is simple.
Moreover, at a generic branch of $\Ga$ at infinity, the local square root $ds$ has a pole of order $2$ with zero residue; hence $ds^2$ has a pole of order $4$.
\end{lemma}

\begin{proof}
Let $t$ be a local parameter at a smooth affine point and write
\[
 ds^2=(x'(t)^2+y'(t)^2)dt^2.
\]
The equation $x'(t)^2+y'(t)^2=0$ says exactly that the tangent direction is isotropic, or equivalently that the tangent line passes through one of the circular points at infinity. Ordinary contact with the corresponding isotropic tangent is therefore equivalent to a simple zero of the factor $x'(t)^2+y'(t)^2$.

At a generic branch at infinity choose a parameter $u$ such that
\[
 x=\frac{a}{u}+O(1),\qquad y=\frac{b}{u}+O(1),
\]
with $a^2+b^2\neq 0$. Then
\[
 dx^2+dy^2=(a^2+b^2)\frac{du^2}{u^4}+O\left(\frac{du^2}{u^3}\right).
\]
Thus a local square root has the form
\[
 ds=\frac{c\,du}{u^2}+O(1)\,du,\qquad c^2=a^2+b^2.
\]
There is no $du/u$ term, so the residue is zero. The asserted pole orders follow.
\end{proof}

Let us also spell out the rational case.  A rational curve
$\Ga\subset \bC P^2$ is given by a triple of polynomials $(U(t),V(t),W(t))$
without a common factor.  If
\[
   x(t)=\frac{U(t)}{W(t)},\qquad y(t)=\frac{V(t)}{W(t)},
\]
then
\[
 ds^2=\left((x'(t))^2+(y'(t))^2\right)dt^2
 = \frac{(U'W-UW')^2+(V'W-VW')^2}{W^4}\,dt^2.
\]
Thus, in genus zero, the existence problem becomes the problem of deciding
when a prescribed rational quadratic differential admits a representation of
this Wronskian form.  The criterion above gives a complementary intrinsic
formulation: one has to choose a factor $dz$ of $q$ with zero residues and zero
periods, and then integrate it to recover the rational function $z=x+iy$.

\section{Exact quadratic and higher order differentials}\label{sec:exact}

For $k \geq 2$, a meromorphic $k$-differential on a Riemann surface $X$ is a meromorphic section of the line bundle $K_{X}^{\otimes k}$ where $K_{X}$ is the canonical line bundle. In local charts, a $k$-differential writes as $f(z)dz^{k}$ where $f$ is a meromorphic function. The case $k = 2$ corresponds to quadratic differentials.
\par
The moduli space of $k$-differentials is stratified according to the multiplicity of the zeros and the poles, see \cite{BCGGM} for background. We denote by $\Omega^{k}\mathcal{M}_{g}(a_{1},\dots,a_{n})$ the moduli space of pairs $(X,\omega)$ where $X$ is a compact Riemann surface of genus $g$ and $\omega$ is a $k$-differential on $X$ with $n$ singularities of multiplicities $a_{1},\dots,a_{n}$ (up to a biholomorphic change of variable). Riemann-Roch theorem implies that $\sum\limits_{i=1}^{n} a_{i} = k(2g-2)$.

\begin{definition} {\rm 
For a (meromorphic) $k$-differential $\phi$ on a compact Riemann surface $X$, the canonical $k$-cover $\Theta :{X^\prime}\longrightarrow X$ is a cover of degree $k$ ramified at the singularities of $\phi$ whose order is not divisible by $k$ and such that the pullbacks of each of the $k$ branches of $\phi^{1/k}$ agree as a uni-valued $1$-form. If the $k$-differential $\phi$ is primitive (if it is not the global power of a differential of lower order), its canonical $k$-cover is connected. The canonical $k$-cover is obtained by continuation of any branch of $\phi^{1/k}$ on $X$. (The canonical $1$-form on $X^\prime$ will be denoted by $\root k \of \phi$.)

In particular, for a meromorphic quadratic differential $\Psi$ on a compact Riemann surface $X$, one gets a canonical double cover $\Theta: X^\prime \to X$ branched at the singularities of $\Psi$ of odd order and a meromorphic $1$-differential $\sqrt \Psi$ on $X^\prime$.}
\end{definition}

\begin{definition} {\rm A meromorphic $k$-differential $\phi$ on a compact Riemann surface $X$ is called {\emph exact} if the $1$-form $\root k \of \phi$ is exact.}
\end{definition}

\begin{theorem}\label{thm:exact-k}
For a meromorphic $k$-differential $\phi$ on a compact Riemann surface $X$, the following three conditions are equivalent:
\begin{enumerate}
\item the canonical $1$-form $\root k \of \phi$ is exact on the canonical $k$-cover $X'$;
\item the Abelian integral $\int \phi^{1/k}$ is an algebraic multivalued function on $X$;
\item there exists a meromorphic function $f$ on $X$ such that
\[
 \phi=f\left(\frac{df}{f}\right)^k.
\]
\end{enumerate}
\end{theorem}

\begin{proof}
If (3) holds, then on the canonical cover one may choose a branch of $f^{1/k}$ and write
\[
 \root k \of \phi=f^{1/k}\frac{df}{f}=k\,d(f^{1/k}),
\]
so the canonical $1$-form is exact. This proves (3)$\Rightarrow$(1), and (1)$\Rightarrow$(2) is immediate: an exact primitive on the finite cover descends to a finite-valued, hence algebraic, multivalued function on $X$.

Conversely, if $\int\phi^{1/k}$ is algebraic, its continuation has only finitely many branches. Therefore the periods of the canonical form on $X'$ vanish; otherwise continuation along powers of a loop with nonzero period would produce infinitely many values. Thus (ii)$\Rightarrow$(i).

Assume (i), and write $\root k \of \phi=dg$ on $X'$. The deck group acts on $dg$ by multiplication by $k$-th roots of unity. After adding a constant to $g$, the same is true for $g$ itself. Hence $h=g^k$ is invariant under the deck group and descends to a meromorphic function on $X$. Since
\[
 dh=k g^{k-1}dg,
\]
one obtains
\[
 \phi=(dg)^k=\frac{1}{k^k}h\left(\frac{dh}{h}\right)^k.
\]
Absorbing the constant $1/k^k$ into $h$ gives (3).
\end{proof}

It should be remarked that any such $k$-differential with exact cover induces a conical metric of constant positive curvature with coaxial monodromy (by considering the logarithmic derivative of the developing map).

\begin{cor}
A singularity of such a $k$-differential $\phi$ either coincides with a singularity of $f$ or a zero of $df$. More precisely, if $z$ is a singularity of order $a$ of $f$, then it is a singularity of order $a-k$ of $\phi$. If $z$ is not a singularity of $f$ but is a zero of order $a$ of $df$, then it is a zero of order $ak$ of $\phi$.

In particular, singularities of $\phi$ whose order is at least $1-k$ and fails to be divisible by $k$ are always zeros of $f$.

It is also true that if $\phi$ has only one pole and is defined on the Riemann sphere (in other words, $\phi$ is a polynomial), then $f$ has no pole except the unique pole of $\phi$ and is therefore also a polynomial.
\end{cor}

\begin{remark}
The latter corollary implies in particular that the only degrees of freedom are the critical values of $f$ (positions of the zeros of the $k$-differential whose order is a multiple of $k$ in the developing map).
\end{remark} 

Theorem~\ref{thm:exact-k} provides a parametrization of $k$-differentials with exact cover on a given compact Riemann surface.

\begin{proposition}\label{prop:exact-k-no-divisible-zeros}
Let $\phi$ be a meromorphic $k$-differential with exact canonical cover on a compact Riemann surface $X$. Assume that $\phi$ has no zero whose order is divisible by $k$. Then $X\simeq \mathbb{CP}^{1}$ and, in a suitable coordinate $z$,
\[
 \phi=\lambda z^{a}dz^{k},\qquad \lambda\in\mathbb C^*,\quad a\in\mathbb Z .
\]
\end{proposition}

\begin{proof}
By Theorem~\ref{thm:exact-k}, there is a meromorphic function $f:X\to\mathbb{CP}^{1}$ such that
\[
 \phi=f\left(\frac{df}{f}\right)^k .
\]
If $p$ is a critical point of $f$ with $f(p)\neq0,\infty$ and ramification order $r\geq1$, then $df$ has a zero of order $r$ at $p$, and therefore $\phi$ has a zero of order $kr$, divisible by $k$. By assumption there are no such points. Hence all ramification of $f$ lies over $0$ and $\infty$.

Thus $f$ restricts to an unramified covering
\[
 X\setminus f^{-1}(\{0,\infty\})\longrightarrow \mathbb{C}^{*}.
\]
Equivalently, after compactification, $f$ is a finite covering of $\mathbb{CP}^{1}$ ramified over at most the two points $0$ and $\infty$. By the Riemann--Hurwitz formula this forces $X\simeq\mathbb{CP}^{1}$, and after choosing a coordinate $z$ on $X$ one may write $f=z^n$ up to multiplication by a non-zero constant and replacement of $z$ by $1/z$. Consequently
\[
 \phi=z^n\left(n\frac{dz}{z}\right)^k
     = n^k z^{n-k}dz^k,
\]
up to a non-zero constant. This is the claimed form.
\end{proof}

\begin{proposition}\label{prop:rational-real-divisor}
Let $\Gamma$ be a rational plane curve parametrized by two real rational
functions $x(t)$ and $y(t)$, and let
\[
 q=(x'(t)^2+y'(t)^2)dt^2
\]
be its arc length quadratic differential on $\mathbb{CP}^1$. Then the
divisor of $q$ is invariant under complex conjugation and every real
zero or pole of $q$ has even order. Equivalently, after multiplication
by a non-zero real constant, the coefficient of $q$ factors as
\[
 \prod_i (t-\alpha_i)^{2a_i}
 \prod_j \bigl(t^2-2u_jt+u_j^2+v_j^2\bigr)^{b_j},
\]
where $\alpha_i,u_j,v_j\in\mathbb R$, $v_j\ne0$, and the exponents
$a_i,b_j$ are integers. In particular, the non-real singularities occur
in conjugate pairs, whereas real singularities occur with even order.
\end{proposition}

\begin{proof}
Set $z(t)=x(t)+iy(t)$ and $w(t)=x(t)-iy(t)$. Then $w(t)=\overline{z(\bar t)}$
and
\[
 q=dz\,dw.
\]
Thus the orders of $dw$ are obtained from the orders of $dz$ by complex
conjugation. Hence the divisor of $q$ is conjugation invariant. If
$t=\alpha$ is real, the orders of $dz$ and $dw$ at $\alpha$ coincide, so
\[
 \operatorname{ord}_{\alpha}(q)=2\operatorname{ord}_{\alpha}(dz).
\]
This proves the parity statement. The displayed factorization is just
the standard factorization of a real rational function into real linear
factors and irreducible real quadratic factors corresponding to
non-real conjugate pairs.
\end{proof}

Quadratic differentials corresponding to straight lines give the flat structure of the usual flat plane, with a unique pole of order four and no other singularities. The following elementary criterion is the form used later.

\begin{proposition}\label{prop:line-qminusfour-safe}
Let $\Gamma$ be a rational plane curve parametrized by two real rational functions $x(t)$ and $y(t)$, and put $z=x+iy$. Assume that the arc length quadratic differential belongs to $\Omega^{2}\mathcal{M}_{0}(-4)$ and that $dz$ has no non-real pole which is cancelled in the product $q=dz\,d\overline z$. Then $\Gamma$ is a straight line.
\end{proposition}

\begin{proof}
After a real M\"obius change of parameter, the unique pole of $q$ is at $t=\infty$. By the hypothesis on cancellations, $dz$ has no finite pole. Hence $z$ is a polynomial. Since $q$ has a pole of order four at infinity, $dz$ has a pole of order two at infinity, so $z$ has degree one. Therefore $x$ and $y$ are affine functions of the real parameter, and the image is a straight line.
\end{proof}

\begin{rem}
A tempting strengthening would be to characterize rational curves whose arc length differential has only real singularities. In this generality one has to control possible cancellations between $dz$ and $dw=d\overline z$ at non-real points. The clean statement used below is the following sufficient criterion: if, after normalization, all singularities of $dz$ itself lie on the real locus, then $dz=\lambda R(t)dt$ with $R$ real rational and $\lambda\in\mathbb C^*$; hence $x'(t)$ and $y'(t)$ are real constant multiples of the same rational function and the image is a line.
\end{rem}

\begin{lemma}\label{lem:real-singularities-even}
Let $X$ be a real algebraic curve with anti-holomorphic involution $p\mapsto\overline p$, and let $x,y$ be real meromorphic functions on $X$. Put
\[
 z=x+iy,\qquad w=x-iy,
\]
so that $w(p)=\overline{z(\overline p)}$, and let $q=dx^2+dy^2=dz\,dw$. Then the divisor of $q$ is invariant under the real structure. At a real point $p$, the order of $q$ is twice the order of $dz$; in particular every real singularity of $q$ has even order.
\end{lemma}

\begin{proof}
The relation $w(p)=\overline{z(\overline p)}$ implies that the orders of $dw$ are obtained from those of $dz$ by conjugation. Since $q=dz\,dw$, the divisor of $q$ is invariant under conjugation. If $p$ is real, then $dz$ and $dw$ have the same order at $p$, so
\[
 \operatorname{ord}_{p}(q)=\operatorname{ord}_{p}(dz)+\operatorname{ord}_{p}(dw)=2\operatorname{ord}_{p}(dz).
\]
This proves the claim.
\end{proof}

\section{Characterization of plane rectifiable curves}\label{sec:main}


The evolute of a planar curve is formed by the centers of its osculating circles. It is equivalently defined as the envelope of the normals of the curve (normal to the tangent line for the Euclidean structure).  The following result reformulates Humbert's theorem in the language of
the arc length quadratic differential.

\begin{theorem}\label{thm:involutes-exact}
Let $X$ be an irreducible algebraic plane curve and let
$\widetilde X$ be its normalization. Assume that $X$ is not a line
and work away from isotropic tangent directions and from the exceptional
points where the usual evolute-involute correspondence degenerates. Let
\[
q=dx^2+dy^2
\]
be the arc length quadratic differential on $\widetilde X$. Then the
following conditions are equivalent:
\begin{enumerate}
\item $X$ is algebraically rectifiable, i.e. the arc length integral
\[
S=\int \sqrt{q}
\]
is an algebraic multivalued function on $X$;
\item the canonical double cover of $q$ is exact;
\item one, equivalently every, non-constant involute of $X$ is an
algebraic curve;
\item $X$ is the evolute of an algebraic curve.
\end{enumerate}
\end{theorem}

\begin{proof}
Let $r=(x,y)$ be a local parametrization of $X$ on the normalization,
and let $S$ be a local primitive of the arc length form, so that
$dS^2=q$. On the canonical double cover of $q$ the unit tangent vector
\[
T=\frac{dr}{dS}
 =
\left(\frac{dx}{dS},\frac{dy}{dS}\right)
\]
has meromorphic coordinates. A local involute of $X$ is given by
\[
I_C=r-(S+C)T,
\]
where $C$ is a constant. If $S$ is algebraic, then all coordinates of
all involutes are algebraic. Conversely, suppose that one non-constant
involute $I_C$ is algebraic. Since $r$ and $T$ are algebraic on the
canonical cover, the vector $(S+C)T=r-I_C$ is algebraic. Away from the
isolated points where $T$ is undefined or zero, at least one coordinate
of $T$ is non-zero; division by that coordinate shows that $S+C$, and
therefore $S$, is algebraic. Thus algebraic rectifiability is equivalent
to the algebraicity of one, or equivalently all, non-constant involutes.

By Theorem~\ref{thm:exact-k} with $k=2$, algebraicity of the Abelian
integral $S=\int\sqrt{q}$ is equivalent to exactness of the canonical
one-form on the canonical double cover of $q$. This proves the
equivalence of (1), (2), and (3).

Finally, outside the standard exceptional locus, the evolute of an
involute of $X$ is $X$ itself. Conversely, if $X$ is the evolute of an
algebraic curve $Y$, then $Y$ is locally an involute of $X$. Hence the
algebraicity of an involute is equivalent to saying that $X$ is the
evolute of an algebraic curve. This is Humbert's geometric
characterization.
\end{proof}

\begin{remark}{\rm
Thus Theorem~\ref{thm:involutes-exact} is the differential-theoretic
form of Humbert's theorem quoted in Proposition~\ref{prop:Hu}. Humbert's
formulation uses evolutes, while the present formulation uses the
canonical cover of the quadratic differential $q=dx^2+dy^2$. The two
viewpoints are equivalent because passing from a curve to its involutes
amounts exactly to integrating the arc length form.}
\end{remark}

\medskip
Using our previous results, we describe algebraically rectifiable rational curves. (Our result generalizes the very special case settled in \cite{SaFa}).  We need to answer the question for which $R_1=U/W$ and $R_2=V/W$ there exists a rational function $f=P/Q$ such that
$$(R_1^\prime)^2+(R_2^\prime)^2=(f^\prime)^2/f.$$

\begin{theorem}\label{thm:two-singularities-circle}
Let $\gamma:\bCP^{1}\to \bCP^{2}$ be a non-constant rational real plane curve, and let
\[
q=dx^{2}+dy^{2}
\]
be its Euclidean arc length quadratic differential on the normalization. Assume that $q$ has exactly two singularities. If the image of $\gamma$ is not a line, then it is a circle. More precisely, after a real M\"obius change of parameter and a Euclidean similarity, the complex coordinate $z=x+iy$ has the form
\[
z(t)=z_{0}+c\left(\frac{t+i}{t-i}\right)^{m},
\]
where $c\neq 0$ and $m\neq 0$ is an integer. In particular, if the parametrization is birational onto its image, then $|m|=1$.
\end{theorem}

\begin{proof}
Put
\[
z=x+iy,\qquad w=x-iy .
\]
Then $q=dz\,dw$, and $w(t)=\overline{z(\bar t)}$. In particular, the divisor of $q$ is invariant under the real structure on the parameter sphere.

We first discuss the case in which the two singularities are real. After a real M\"obius change of parameter, they are $0$ and $\infty$. Since, at a real point, $dz$ and $dw$ have the same order, the orders of $q$ at both singularities are even. Thus
\[
q=Ct^{2a}dt^{2}
\]
for some integer $a$ and some $C\neq0$. Correspondingly, $dz=\lambda t^{a}dt$ after changing the non-zero constant $\lambda$. If $a=-1$, then $z$ has a logarithmic term, impossible for a rational parametrization. If $a\neq -1$, then
\[
z=z_{0}+ct^{a+1}.
\]
Since $w=\overline{z(\bar t)}$, the real parametrization is contained in the real affine line determined by the complex direction $c$. Hence in the real case the image is a line.

It remains to consider the case in which the two singularities are non-real conjugate points. Since the divisor of a quadratic differential on $\bCP^1$ has degree $-4$, the two conjugate singularities have the same order and therefore both are double poles. After a real M\"obius change of parameter, they are $i$ and $-i$, and
\[
q=C\frac{dt^{2}}{(t^{2}+1)^{2}}
\]
for some $C\neq0$. Since $q=dz\,dw$, there is an integer $a$ and non-zero constants $\lambda,\mu$ such that
\[
dz=\lambda\frac{(t+i)^{a}}{(t-i)^{a+2}}dt,\qquad
 dw=\mu\frac{(t-i)^{a}}{(t+i)^{a+2}}dt .
\]
Set
\[
s=\frac{t+i}{t-i}.
\]
Then $s$ maps the real line to the unit circle and
\[
dz=-\frac{\lambda}{2i}s^{a}ds .
\]
The case $a=-1$ would give a logarithm, so it is impossible for a rational parametrization. Therefore
\[
z=z_{0}-\frac{\lambda}{2i(a+1)}s^{a+1}.
\]
For real $t$ one has $|s|=1$, so the image is contained in a Euclidean circle. Since the curve is irreducible and non-constant, its image is this circle. Taking $m=a+1$ gives the asserted normal form. Finally, a birational parametrization of the circle has degree one, so $|m|=1$.
\end{proof}

The simplest quadratic differentials with exact cover are of the form $\lambda t^{a}dt^{2}$. Such differentials can be obtained as arc length quadratic differentials of some plane curve only in the case $a=0$ and $a=-2$. This defines respectively the straight line and the circle. In the second case, the quadratic differential does not have exact cover.

\section{Examples}\label{sec:examp}

The examples below are meant to illustrate distinct phenomena rather than to
provide a long catalogue.  Conics give the first non-trivial strata of the
arc-length quadratic differential; the Bernoulli lemniscate gives a finite-area
``pillowcase'' differential; the semicubic parabola gives a classical
rectifiable singular curve; and the last two rational examples show,
respectively, a higher degree Pythagorean-hodograph curve and a more generic
rational curve with several simple zeros of the length differential.

\subsection{Conics}

Let $X\simeq\mathbb{CP}^{1}$ be a non-degenerate conic, parametrized by two
meromorphic functions $f$ and $g$.  The Euclidean arc-length quadratic
differential on the normalization is
\[
        q=df^{2}+dg^{2}.
\]
For a general conic this differential is not exact on its canonical double
cover; the circle is the exceptional rectifiable case.

\subsubsection{Hyperbolas}

Put
\[
        f(t)=a\frac{1+t^{2}}{2t},\qquad
        g(t)=b\frac{1-t^{2}}{2t}.
\]
Then $(f/a)^2-(g/b)^2=1$, and
\[
        q=\frac{a^{2}(t^{2}-1)^{2}+b^{2}(t^{2}+1)^{2}}{4t^{4}}\,dt^{2}.
\]
Thus $q$ has poles of order four at $0$ and $\infty$.  Its zeros are the four
simple roots of
\[
        t^{4}+2\frac{b^{2}-a^{2}}{a^{2}+b^{2}}t^{2}+1=0,
\]
namely
\[
        t=\pm\frac{a+bi}{\sqrt{a^{2}+b^{2}}},\qquad
        t=\pm\frac{a-bi}{\sqrt{a^{2}+b^{2}}}.
\]
Consequently a non-circular hyperbola belongs to
$\Omega^{2}\mathcal{M}_{0}(1,1,1,1,-4,-4)$.

\subsubsection{Ellipses}

Parametrize an ellipse by
\[
        f(t)=a\frac{2t}{1+t^{2}},\qquad
        g(t)=b\frac{1-t^{2}}{1+t^{2}}.
\]
Then $(f/a)^2+(g/b)^2=1$, and
\[
        q=\frac{4a^{2}(1-t^{2})^{2}+16b^{2}t^{2}}{(1+t^{2})^{4}}\,dt^{2}.
\]
If $a=b$, the ellipse is a circle and
\[
        q=\frac{4a^{2}}{(1+t^{2})^{2}}\,dt^{2}
          =\left(\frac{2a}{1+t^{2}}\,dt\right)^{2}.
\]
It belongs to $\Omega^{2}\mathcal{M}_{0}(-2,-2)$.  The square root has
non-zero periods around the two poles, so the corresponding arc-length
integral is logarithmic on the complex normalization, although the real circle
has the expected elementary length parameter.

If $a\ne b$, then $q$ has two poles of order four at $t=\pm i$.  Its zeros are
the roots of
\[
        t^{4}+\frac{4b^{2}-2a^{2}}{a^{2}}t^{2}+1=0.
\]
For a non-circular ellipse these four roots are simple, and the differential
again lies in $\Omega^{2}\mathcal{M}_{0}(1,1,1,1,-4,-4)$.

\subsubsection{Parabolas}

For the parabola $f^{2}=ag$ we use the parametrization
\[
        f(t)=t,\qquad g(t)=\frac{t^{2}}{a}.
\]
Then
\[
        q=\frac{4t^{2}+a^{2}}{a^{2}}\,dt^{2}.
\]
It has two simple zeros at $t=\pm ai/2$ and one pole of order six at infinity;
hence it belongs to $\Omega^{2}\mathcal{M}_{0}(1,1,-6)$.

\subsection{Rational curves beyond conics}

\subsubsection{The Bernoulli lemniscate}

The Bernoulli lemniscate
\[
        (x^{2}+y^{2})^{2}=x^{2}-y^{2}
\]
is parametrized by
\[
        x(t)=\frac{t+t^{3}}{1+t^{4}},\qquad
        y(t)=\frac{t-t^{3}}{1+t^{4}}.
\]
A direct calculation gives
\[
        q=dx^{2}+dy^{2}=\frac{2}{t^{4}+1}\,dt^{2}.
\]
Thus $q$ has four simple poles and no zeros; it belongs to
$\Omega^{2}\mathcal{M}_{0}(-1,-1,-1,-1)$.  The associated flat surface is the
standard pillowcase-type surface.

\subsubsection{The semicubic parabola}

The semicubic parabola
\[
        ay^{2}=x^{3}
\]
is parametrized by
\[
        x(t)=at^{2},\qquad y(t)=at^{3}.
\]
Its arc-length quadratic differential is
\[
        q=a^{2}t^{2}(4+9t^{2})\,dt^{2}.
\]
It belongs to $\Omega^{2}\mathcal{M}_{0}(2,1,1,-8)$.  Moreover $q$ is exact in
the sense of Theorem~\ref{thm:exact-k}, since
\[
        q=f\left(\frac{df}{f}\right)^{2},\qquad
        f(t)=\frac{a^{2}}{2916}(4+9t^{2})^{3}.
\]
This is the classical singular algebraically rectifiable example.

\subsubsection{A Pythagorean-hodograph quintic}

Let
\[
        u(t)=t^{2}-1,\qquad v(t)=2t,
\]
and define a real polynomial curve by
\[
        x'(t)=u(t)^{2}-v(t)^{2}=t^{4}-6t^{2}+1,
        \qquad
        y'(t)=2u(t)v(t)=4t^{3}-4t.
\]
After integration one obtains
\[
        x(t)=\frac{t^{5}}5-2t^{3}+t,
        \qquad
        y(t)=t^{4}-2t^{2}.
\]
Then
\[
        q=dx^{2}+dy^{2}=(u^{2}+v^{2})^{2}\,dt^{2}=(t^{2}+1)^{4}\,dt^{2}.
\]
Thus the curve is algebraically rectifiable: the arc length form is
$(t^{2}+1)^{2}dt$, whose primitive is rational.  The differential belongs to
\[
        \Omega^{2}\mathcal{M}_{0}(4,4,-12),
\]
with zeros of order four at $t=\pm i$ and a pole of order twelve at infinity.
This example is useful because it is smooth as a parametrized polynomial curve
and rectifiable for a reason different from the elementary semicubic parabola:
it comes from a non-linear Pythagorean-hodograph identity.

\subsubsection{A rational curve with many zeros}

Consider the rational curve
\[
        x(t)=\frac{t^{3}}{(t^{2}+1)^{2}},\qquad
        y(t)=\frac{3t^{2}+1}{(t^{2}+1)^{2}}.
\]
A direct calculation gives
\[
        q=\frac{t^{2}(t^{6}+30t^{4}-15t^{2}+4)}{(t^{2}+1)^{6}}\,dt^{2}.
\]
The points $t=\pm i$ are poles of order six, $t=0$ is a zero of order two,
and, provided the sextic factor has simple roots, the remaining six zeros are
simple.  Hence this example lies in
\[
        \Omega^{2}\mathcal{M}_{0}(2,1,1,1,1,1,1,-6,-6).
\]
It is included to show the more typical behaviour of rational plane curves:
the length differential may have several zeros and poles and need not be the
square of a rational one-form.

\section{Affine differential geometry of plane curves and affine evolute}\label{sec:aff}

In Euclidean geometry, the osculating circle at a point of a plane curve is the unique circle making second order contact (limit of the circle passing through three infinitesimally close points). This definition can be extended to more general objects of affine geometry.

\begin{definition}
For a point $x$ of a plane curve $\Gamma$, the \textit{hyperosculating parabola} of $\Gamma$ at $x$ is the unique parabola making third order contact with $\Gamma$ at $x$. Similarly, the \textit{hyperosculating conic} of $\Gamma$ at $x$ is the unique conic making fourth order contact.
\end{definition}

\subsection{Hyperosculating parabola}

For a curve $\Gamma$, between any pair of points $x$ and $y$, there is a triangle whose vertices are $x$, $y$ and the intersection of the tangent lines of $\Gamma$ at $x$ and $y$. The area of this triangle is invariant by the action of the special-affine group.

There is a unique parabola passing through $x$ and $y$ with the same tangent lines as $\Gamma$ at these two points. This provides an equivalent definition of hyperosculating parabola.

Infinitesimally close points $x$ and $y$ define an infinitesimal triangle whose area is a special-affine local invariant. For this purpose, curve $\Gamma$ can be locally identified with its hyperosculating parabola.

For a parabola parametrized by $\gamma(t)$, the exterior product $\gamma'(t) \wedge \gamma''(t)$ is constant to a bivector $A$ (moment of a central force in classical mechanics). The area of the triangle drawn by two points of parameters $t_{1}$ and $t_{2}$ is (after some calculations) $\frac{(t_{2}-t_{1})^{3}A}{8}$.

As a consequence, the third root of this area is additive for parameter $t$ and can serve as an affine arc length on the parabola.

By identifying locally any smooth curve with its hyperosculating parabola, we get an arc length function by integrating the third root of $\gamma'(t) \wedge \gamma''(t)dt^{3}$. Just like the Euclidean structure of the plane induces a quadratic differential, the special-affine structure of the plane induces a cubic differential.

The cubic differential induced by the special-affine structure is an invariant of a plane embedding. As in the Euclidean case, however, one should not claim uniqueness of the embedding from this differential alone without specifying the appropriate affine curvature data. We therefore use the cubic differential below as a computable invariant rather than as a complete invariant.

\subsection{Affine evolute}

At every point $x$ of a planar curve $\Gamma$, the direction of the \textit{affine normal line} to $\Gamma$ at $x$ is the direction of the directrix of the hyperosculating parabola of $\Gamma$ at $x$.

\begin{definition}
The \textit{affine evolute} of a smooth plane curve $\Gamma$ is the envelope of its affine normal lines. Equivalently, in the classical affine differential geometry formulation, it is the curve traced by the centers of the hyperosculating conics of $\Gamma$.
\end{definition}

The center of a conic is a well-defined affine notion. For an ellipse it is the intersection point of any diameter. For an hyperbola it is the intersection of the two asymptotic lines.

\begin{example}
The affine evolute of any ellipse or hyperbola is its center. The affine evolute of a parabola is a point at infinity.
\end{example}

The first interesting cases begin with a cubic curve.

\begin{proposition}
The affine evolute of an algebraic plane curve is algebraic, wherever it is defined.
\end{proposition}

\begin{proof}
On the normalization of the curve, the affine tangent, affine normal, and the center of the hyperosculating conic are obtained from a finite jet of a rational parametrization by algebraic operations and differentiation. Hence the corresponding point of the affine evolute is given by rational functions on the normalization, away from the exceptional locus where the affine normal construction degenerates. The image of a rational map from an algebraic curve is algebraic.
\end{proof}

In both cases, the vector between a point of the evolute and the corresponding point of an involute is colinear to the tangent vector. Therefore, this correspondence is birational if and only some quantity is rational. This quantity is the ratio between the arc length and a square root.

\subsection{Cubic differential for rational parametrizations}

For a real algebraic curve parametrized by $OM(t)=(f(t),g(t))$, cubic differential $OM'(t) \wedge OM''(t)dt^{3}$ is a special affine invariant, see \cite{COT}.

For straight lines, the cubic differential is zero.

For a parabola, the cubic differential is $adt^{3}$. It has a unique pole of order six.

For an ellipse parametrized by $f(t)=a\frac{2t}{1+t^{2}}$ and $g(t)=b\frac{1-t^{2}}{1+t^{2}}$, the cubic differential is $\frac{-8ab}{(1+t^{2})^{3}}dt^{3}$. The associated flat surface is an infinite cylinder, just like the flat surface associated by the quadratic differential induced on the circle by the ambient Euclidean structure.

For a hyperbola parametrized by $f(t)=a\frac{1+t^{2}}{2t}$ and $g(t)=b\frac{1-t^{2}}{2t}$, the cubic differential is $\frac{ab}{t^{3}}dt^{3}$. The associated flat surface is also an infinite cylinder. It is the same as the ellipse up to a rotation.

For the semicubic parabola parametrized by $f(t)=at^{2}$ and $g(t)=at^{3}$, cubic differential is $6a^{2}t^{2}dt^{3}$. It is far simpler than the quadratic differential of Euclidean geometry. It is a cubic differential with a zero of order two and a pole of order eight. In particular, it has exact cover.

\begin{proposition}\label{prop:affine-involutes-exact}
Let $X$ be the normalization of a non-linear real algebraic plane curve and let
\[
\Phi=(\gamma'(t)\wedge\gamma''(t))dt^{3}
\]
be the cubic differential induced by the special-affine structure of the plane.
Assume that the affine involutes are understood in the classical special-affine
sense on the open set where the affine normal construction is non-degenerate.
Then the affine involutes are algebraic if and only if the cubic differential
$\Phi$ has exact canonical cover.
\end{proposition}

\begin{proof}
On the smooth non-degenerate locus, the affine arc length is obtained by
integrating a local branch of $\Phi^{1/3}$. The affine normal, affine evolute
and corresponding affine involutes are obtained from a finite jet of the
parametrization by algebraic operations, except for this affine arc length
parameter. Therefore the coordinates of the affine involutes are algebraic over
the function field of $X$ precisely when the Abelian integral
\[
\int \Phi^{1/3}
\]
is an algebraic multivalued function on $X$.

By Theorem~\ref{thm:exact-k}, applied with $k=3$, this condition is equivalent
to exactness of the canonical one-form on the canonical cubic cover of $X$
associated with $\Phi$.
\end{proof}

\begin{example}
Bernoulli lemniscate, defined by equation $(x^{2}+y^{2})^{2}=x^{2}-y^{2}$ is parametrized by $x(t)=\frac{t+t^{3}}{1+t^{4}}$ and $y(t)=\frac{t-t^{3}}{1+t^{4}}$. The arc length quadratic differential simplifies to $\frac{2dt^{2}}{t^{4}+1}$. This is an example of pillowcase flat surface. It belongs to $\Omega^{2}\mathcal{M}_{0}(-1,-1,-1,-1)$.

From an affine point of view, the cubic differential simplifies to $\frac{-12tdt^{3}}{(t^{4}+1)^{2}}$. Such cubic differential belongs to $\Omega^{3}\mathcal{M}_{0}(1,1,-2,-2,-2,-2)$. Zero and infinity are simple zeros while the fourth roots of $-1$ are double poles. The associated flat surface has finite area.
\end{example}

\begin{example}
Singular cubic curve, defined by equation $y^{2}=x^{2}+x^{3}$ is parametrized by $x(t)=t^{2}-1$ and $y(t)=t^{3}-t$. The arc length quadratic differential is $(9t^{4}-2t^{2}+1)dt^{2}$. It belongs to $\Omega^{2}\mathcal{M}_{0}(1,1,1,1,-8)$.

From an affine point of view, the cubic differential is $(6t^{2}+1)dt^{3}$. Such cubic differential belongs to $\Omega^{3}\mathcal{M}_{0}(1,1,-8)$.
\end{example}

\begin{example}
For any real rational function $R$, real algebraic curve defined by equation $y=R(x)$ is parametrized by $x(t)=t$ and $y(t)=R(t)$. The arc length quadratic differential is $(1+R'^{2})dt^{2}$. However, the affine cubic differential is $R''(t)dt^{3}$. Any second derivative of a real rational function can appear in that way. In particular, for a generic rational function $R$, zeros of the cubic differential are simple and its poles are of order three (with nonzero cubical residues).
\end{example}

\section{Appendix. Pl\"ucker formulas for generic (rational) curves and their evolutes}

The next information is either standard or can be easily derived from the results of Theorem 54 of  \cite{Coo}.

\begin{lemma}\label{lm:rat}  {\rm (1)} A generic plane projective curve $\ga\subset \bC P^2$ of degree $d$ has genus $\binom {d-1}{2}$, class  $d^\prime= d(d-1)$ (i.e. the degree of the dual curve $\ga^\ast\subset (\bC P^2)^\ast$) and  $\sharp_i=3d(d-2)$ inflection points.  The dual curve $\ga^\ast$ has $\sharp^\prime_{cusps}=3d(d-2)$ cusps, no inflection points and $\sharp^\prime_{nodes}=\binom{d(d-1)-1}{2}-\binom{d-1}{2}$ nodes.

\noindent
{\rm (2)} A generic plane rational projective curve $\ga\subset \bC P^2$ of degree $d$ is nodal with $\sharp_{nodes}=\binom{d-1}{2}$ nodes and $\sharp_i=3(d-2)$ inflection points. The class $d^\prime$ of such curve  equals $2d-2$. The dual curve $\ga^\ast$ has $\sharp^\prime_{cusps}=3(d-2)$ cusps, no inflection points and $\sharp^\prime_{nodes}=2(d-2)(d-3)$ nodes.
\end{lemma}

\begin{proof} Item (1) is standard. To settle item (2), notice that a generic plane rational curve $\gamma$ of degree $d$ is nodal, this is a standard genericity statement. Since the genus of $\gamma$ is $0$ its number of nodes equals its expected genus which for a curve of degree $d$ equals $\binom {d-1}{2}$. The number $\sharp_i$ of inflection points of $\gamma$ equals the expected number of inflection points minus $6\sharp_{nodes}$. The expected number of inflection points equals $3d(d-2)$. Thus $\sharp_i$ of  $\gamma$  equals $3d(d-2)-3(d-1)(d-2)=3d-6$.
 Furthermore, for a plane curve of degree $d$ with cusps and nodes only, its class $d^\prime$ equals $d(d-1)-2\sharp_{nodes} -3\sharp_{cusps}$. Therefore in our case, $d^\prime= d(d-1)-(d-1)(d-2)=2d-2$. Since a generic rational curve $\gamma$ has no cusps its dual $\gamma^\prime$ has no inflection points. The dual $\gamma^\prime$ has $\sharp^\prime_{cusps}=3(d-2)$ cusps and genus $0$. Therefore its number $\sharp^\prime_{nodes}$  of nodes equals $\binom{2d-3}{2}-3d+6=\frac{(2d-3)(2d-4)}{2}-3d+6=2d^2-10d+12=2(d-2)(d-3)$.
 \end{proof}

  \begin{lemma} {\rm (1)}  For a generic plane curve $\ga$ of degree $d$,  its evolute $E_\ga\subset \bC P^2$  is a plane curve of degree $d_E=3d(d-1)$ and class $d^\prime_E=d^2$. $E_\ga$ has no inflection points,
  $\sharp_{cusps}^E= d(6d-9)$ cusps and $\sharp^E_{nodes}=\frac{d (3d-5)(3d^2-d-6)}{2}$ nodes.   The number $\sharp_{nodes}^N$ of nodes of $\tilde N_\Ga^\bC$  equals $\binom{d^2-1}{2} - \binom{d-1}{2}=\frac{d(d-1) (d^2+d-3)}{2}$.  There are no  cusps on $\tilde N_\Ga^\bC$   (since $\tilde E_\Ga$ has no inflection points).

 \noindent
  {\rm (2)}
  For a generic rational curve $\ga$ of degree $d$, its evolute $E_\ga\subset \bC P^2$  is a rational curve of degree $d_E=6(d-1)$ and class $d^\prime_E=3d-2$. $E_\ga$ has no inflection points,  $\sharp_{cusps}^E=3(3d-4)$ cusps and $\sharp^E_{nodes}=2(3d-4)(3d-5)$ nodes. The  curve of normals $N_\ga\subset (\bC P^2)^\ast$ (i.e. the dual curve to the evolute $E_\ga$)  has no cusps, $\binom {3d-3}{2}$ nodes and $3(3d-4)$ inflection points.

\end{lemma}

\begin{proof} A generic rational curve $\ga$ satisfies the assumptions of Theorem 54 of  \cite{Coo}. Thus, by Theorem 54,  the degree $d_E$ of the evolute $E_\ga$ is given by $d_E=3d+\sharp_i=3d+3(d-2)=6(d-1)$.  The class $d^\prime_E$ of $E_\ga$ equals $d+d^\prime=d+2d-2=3d-2$.   Thus the curve of normals $N_\ga$ has degree $3d-2$ and genus $0$. Since, by Theorem 54, the evolute $E_\ga$ has no inflection points the curve $N_\ga$ has no cusps. Thus $N_\Ga$ has $\sharp^N_{nodes}=\binom {3d-3}{2}$ nodes. Finally, its number $\sharp^N_{i}$ of inflection points equals $3(3d-2)(3d-4)-6 \sharp^N_{nodes}=9d-12$.
\end{proof}

\smallskip
\noindent
\emph{Acknowledgements.} We are sincerely grateful to Professor E.~Shustin for his interest in this project and discussions. The second author is supported by the Beijing Natural Science Foundation (Grant IS23005) and the French National Research Agency under the project TIGerS (ANR-24-CE40-3604).

\end{document}